\documentclass{amsart}

\usepackage[utf8]{inputenc}
\usepackage{amsmath, amssymb,amsthm,tikz}
\usepackage{tikz-cd}
\usepackage{tikz}
\usepackage{mathrsfs}
\usepackage{extarrows}
\usepackage{graphicx}
\usepackage{thmtools}
\usepackage{mathtools}
\usepackage{hyperref}

\usepackage[english]{babel}

\usepackage[toc,page]{appendix}

\newcommand{\Z}{\mathbb{Z}}

\newcommand{\Q}{\mathbb{Q}}

\newcommand{\A}{\mathbb{A}}
\newcommand{\mcA}{\mathcal{A}}

\newcommand{\Oo}{\mathcal{O}}

\newcommand{\End}{\operatorname{End}}

\newcommand{\Image}{\operatorname{Im}}

\newcommand{\Id}{\operatorname{Id}}

\newcommand{\Spec}{\operatorname{Spec}}

\newcommand{\Hom}{\operatorname{Hom}}

\newcommand{\Sh}{\operatorname{Sh}}

\newcommand{\GL}{\operatorname{GL}}

\newcommand{\GSp}{\operatorname{GSp}}

\hypersetup{
    colorlinks,
    citecolor=blue,
    filecolor=blue,
    linkcolor=blue,
    urlcolor=blue
}

\renewcommand{\tilde}{\widetilde}
\renewcommand{\mod}{\operatorname{mod}}

\title{On the Hodge embedding for PEL type integral models of Shimura varieties}
\author{Yujie Xu}
\address{Department of Mathematics, Harvard University}
\curraddr{}
\email{yujiex@math.harvard.edu}
\date{}

\numberwithin{equation}{section}

\newtheorem{prop}[equation]{Proposition}

\newtheorem{lem}[equation]{Lemma}
\newtheorem{Coro}[equation]{Corollary}

\theoremstyle{definition}
\newtheorem{Defn}[equation]{Definition}
\newtheorem{remark}[equation]{Remark}

\begin{document}

\maketitle

\begin{abstract}
We give a simple proof 
that Kottwitz's PEL type integral models of Shimura varieties admit closed embeddings into Siegel integral models. We also show that Rapoport's and Kottwitz's integral models agree with Kisin's integral models for relevant Shimura data. 
\end{abstract}

\tableofcontents

\section{Introduction}
Let $\Sh_K(G,X)$ be a Shimura variety of PEL type (in particular, one can take it to be of Hilbert modular type) and $\mathcal{S}_K(G,X)$ its integral model \cite{Kottwitz} (resp.~\cite{Rapoport}). The natural morphism $\mathcal{S}_K(G,X)\to\mathscr{S}_{K'}(\GSp,S^{\pm})$, where the target is some suitable Siegel integral model, is often called the \textit{Hodge morphism}. 

We give a simple, down-to-earth proof that the Hodge morphism is a closed embedding for PEL type integral models. Along the way, we also prove equivalence of the Kottwitz and Rapoport models with Kisin's integral models, constructed by taking the flat closure of the generic fibre inside the Siegel integral models.\footnote{The same proofs go through for parahoric integral models, but our expositions shall focus on the hyperspecial cases.}

In terms of moduli descriptions, the Hodge morphism is given by forgetting the $\Oo_B$-action on an abelian scheme $\mcA$, where $B$ is a semisimple $\Q$-algebra attached to the PEL moduli problem. Let $T^{(p)}(\mcA)$ be the prime-to-$p$ Tate module. For a point on $\mathcal{S}_K(G,X)$, the corresponding level structure $\eta:V\otimes \A_f^p\xrightarrow{\sim}T^{(p)}(\mcA)\otimes\A_f^p$ is compatible with the $\Oo_B$-actions. This shows that the $\Oo_B$-action on $\mcA$ is already determined by $\eta$ (see Proposition \ref{injectivityprop-PEL}), and hence that the Hodge morphism is an embedding (Proposition \ref{infinite-level-embedding}). Strictly speaking, this argument only applies when we take the infinite 
level structure away from $p$, but it is not hard to descent down to finite levels (see \ref{descent-down-section}). 

\section{Preliminaries}
First we recall the basic theory. 
Let $B$ be a  
semi-simple $\Q$-algebra, endowed with a positive involution $*$. Let $V$ be a finite dimensional $B$-module $V$, endowed with a non-degenerate bilinear alternating pairing $\langle\cdot,\cdot\rangle$. 
Let $G$ be the reductive group over $\Q$ defined by
\[G(R)=\{g\in \GL(V\otimes R),\exists\mu\in R^*,\forall x,y\in V\otimes R, \langle gx,gy\rangle=\mu\langle x,y\rangle\}\]
Let $E=E(G,X)$ be the reflex field of the Shimura datum $(G,X)$. Let $\Oo_{E,(v)}$ be the localization of $\Oo_E$ at a prime $v$ of $E$ above $p$.

Let $\Oo_B$ be a $\Z_{(p)}$-order 
in $B$ that is stable under the involution $*$ of $B$ and becomes maximal after tensoring with $\Z_p$.
We also require that:
\begin{itemize}
    \item $B$ is unramified at $p$, i.e. $B_{\Q_p}=B\otimes_{\Q}\Q_p$ is isomorphic to a product of matrix algebras over unramified extension of $\Q_p$.
    \item There exists a $\Z_p$-lattice $\Lambda$ in $V_{\Q_p}$ that is stable under $\Oo_B$, and such that the pairing $\langle\cdot,\cdot\rangle$ induces a perfect duality of $\Lambda$ with itself.
\end{itemize}
We fix such a $\Lambda$ (our level structure will be defined in terms of this lattice, however, the construction of the integral model does not depend on the choice of $\Lambda$).

Let $S$ be a scheme over $\Spec\Oo_{E,(v)}$.
We define an $R$-isogeny between abelian schemes $\mcA$ and $\mcA'$ to be an isomorphism in the localized category where the objects are abelian schemes and the set of morphisms from $\mcA$ to $\mcA'$ is $\Hom(\mcA,\mcA')\otimes_{\Z} R$. In particular, an $R$-polarization of $\mcA$ is a polarization of $\mcA$ that is also an $R$-isogeny from $\mcA$ to the dual abelian scheme $\mcA^t$. 

\begin{Defn}
We define a $\Z_{(p)}$-polarized abelian scheme (over $S$) with an action of $\Oo_B$ to be a tuple $(\mcA,\lambda,\iota)$ where:
\begin{itemize}
    \item $\mcA$ is an abelian scheme over $S$.
    \item $\lambda$ is a $\Z_{(p)}$-polarization.
    \item $\iota$ is an injective ring homomorphism $\Oo_B\to\End_S(\mcA)\otimes_{\Z}\Z_{(p)}$ which respects involution on both sides (the involution $*$ on $\Oo_B$ inherited from that on $B$, and the Rosatti involution $\dagger$ coming from the polarization $\lambda$ on $\End_S(\mcA)$.)
\end{itemize}
\end{Defn}

Take any geometric point $s$ of $S$. 
Consider the usual Tate module 
$T(\mcA_s)=\underset{N}{\varprojlim}\mcA_s[N]$, the \textit{prime-to-$p$ Tate module} 
$T^{(p)}(\mcA_s)=\underset{p\nmid N}{\varprojlim}\mcA_s[N]=T(\mcA_s)\otimes_{\hat{\Z}}\hat{\Z}^{(p)}$ 
and the \textit{rational Tate module}
$V^{(p)}(\mcA_s)=H_1(\mcA_s,\A_f^p)=T^{(p)}(\mcA_s)\otimes_{\hat{\Z}^{(p)}}\A_f^p$. 
We choose a geometric point $s$ of a connected component of $S$.
\begin{Defn}\label{PEL-infin-level}
The infinite level strucure (away from $p$) 
is a morphism
\[\eta_{\infty}^p: \Lambda\otimes\A_f^p\xrightarrow{\sim} V^{(p)}(\mcA_s)\] which respects the bilinear forms up to a scalar in $(\A_f^p)^*$ and is compatible with the $\Oo_B$-action on both sides.
\end{Defn}

\begin{Defn}  
Let $K^p\subset G(\A_f^p)$ be a compact open subgroup. A $K^p$-level structure on $(\mcA,\lambda,\iota)$ is a $K^p$-orbit $\overline{\eta}^p$ of the morphism $\eta_{\infty}^p$ given in Defintion \ref{PEL-infin-level} such that the orbit is fixed under the action of $\pi_1(s,S)$.
\end{Defn}
Note that a level structure is independent of the choice of $s$.

\begin{Defn}\cite{Kottwitz}  
We define the moduli functor $\mathfrak{F}_{K^p}$ from the category of 
$\Spec\Oo_{E,(v)}$-schemes to the category of sets:
\[S\mapsto \mathfrak{F}_{K^p}(S):=\{(\mcA/S, \lambda,\iota,\overline{\eta}^p)\}/\sim\]
\begin{itemize}
    \item $(\mcA/S,\lambda,\iota)$ is a $\Z_{(p)}$-polarized abelian scheme\footnote{I assume all ``abelian schemes'' are automatically ``projective'' by definition, so I'm dropping ``projective'' from the statement} over $S$, with an action of $\Oo_B$ which respects the determinant condition \cite[$\mathsection$5]{Kottwitz} 
    \item $\overline{\eta}^p$ is a $K^p$-level structure over each connected component of $S$
\end{itemize}
\end{Defn} 
\begin{prop}\cite[$\mathsection 5$]{Kottwitz} 
For sufficiently small $K^p$, the moduli functor $\mathfrak{F}_{K^p}$ is representable by a quasi-projective scheme $\mathcal{S}_{K^p}$ over $\Oo_E\otimes_{\Z}\Z_{(p)}$.
\end{prop}

\begin{Defn}\cite{Kottwitz} 
We define the moduli functor $\mathfrak{F}_{\infty}$ from the category of 
$\Spec\Oo_{E,(v)}$-schemes 
to the category of sets:
\[S\mapsto \mathfrak{F}_{\infty}(S):=\{(\mcA/S, \lambda,\iota,\eta^p)\}/\sim\]
\begin{itemize}
    \item $(\mcA/S,\lambda,\iota)$ is a $\Z_{(p)}$-polarized abelian scheme
    over $S$, with an action of $\Oo_B$ which respects the determinant condition \cite[$\mathsection$ 5]{Kottwitz} 
    \item $\eta^p$ is the infinite level structure as defined in Definition \ref{PEL-infin-level}. 
\end{itemize}
\end{Defn}
From the definition, we clearly have $\mathfrak{F}_{\infty}=\underset{K^p}{\varprojlim}\mathfrak{F}_{K^pK_p}$, which is representable by a 
$\Spec\Oo_{E,(v)}$-scheme $\mathcal{S}_{\infty}^{(p)}:=\underset{K^p}{\varprojlim}\mathcal{S}_{K^p}$. 
We denote by $\mathscr{S}_{K_{\infty}}^{(p)}(\GSp,S^{\pm})$ the Siegel moduli space at infinite level away from $p$. 
\section{Hodge morphism and Equivalence of Models}\label{equivalence-of-models-section}
\subsection{Equivalence of Models}
In this section, we show that Rapoport \cite{Rapoport}, Kottwitz \cite{Kottwitz} and Kisin models \cite{Kisin-integral-model} agree, for appropriately chosen Shimura data. 
Let $(G,X)$ be a Shimura datum of PEL type (resp.~Hilbert modular type) 
with integral model 
$\mathcal{S}_{K}(G,X)$ in the sense of Kottwitz (resp.~Rapoport). 
First we make the following remark.
\begin{lem}\label{Kottwitz-properness-finite}
The Hodge morphism 
$\Phi_{K}:\mathcal{S}_K(G,X)\to \mathscr{S}_{K'}(\GSp,S^{\pm})$ 
is finite. 
\end{lem}
\begin{proof}
By \cite[p.391]{Kottwitz}, the morphism $\Phi_K$ is already quasi-projective. By the valuative criterion for properness, $\Phi_K$ is also proper, thus $\Phi_K$ is projective. 
On the other hand, since $\Phi_K$ has discrete fibres (given by the endomorphism structures that can be endowed on a principally polarized abelian scheme), thus $\Phi_K$ is a finite morphism.
\end{proof}
We give a different proof for this Lemma in \cite{fppf-finiteness}, using finiteness of fppf cohomology. Moreover, in the following sections (see Corollary \ref{finite-level-embedding}), we show a stronger result that the finite morphism $\Phi_K$ is in fact a closed embedding. Note that the Hodge morphism for a general Hodge type Shimura variety is the subject of study in \cite{xu2020normalization}. 

\begin{lem}\label{equiv-models}
When Shimura datum $(G,X)$ is of Hilbert modular type, the Kisin model $\mathscr{S}_{K_N}(G,X)$ agrees with Rapoport's integral model $\mathfrak{M}_N$. 

Likewise, when Shimura datum $(G,X)$ is of PEL type, the Kisin model agrees with Kottwitz's model $\mathcal{S}_K(G,X)$. 
\end{lem}
\begin{proof}
By Lemma \ref{Kottwitz-properness-finite}, the Hodge morphism 
\[\Phi_N:\mathfrak{M}_N\to \mathscr{S}_{K'_N}(\GSp,S^{\pm})\]
is a finite morphism. We denote by $(G,X)$ the Hilbert modular Shimura datum corresponding to the integral model $\mathfrak{M}_N$. Since $\mathscr{S}_{K_N}^-(G,X)$ is the closure of $\Sh_{K_N}(G,X)$ inside $\mathscr{S}_{K'_N}(\GSp,S^{\pm})$, $\Phi_N$ induces a map
\[\Phi_N^-:\mathfrak{M}_N\to \mathscr{S}_{K_N}^-(G,X)\]
which maps each irreducible component to an irreducible component, and moreover is birational, because $\mathfrak{M}_N$ and $\mathscr{S}_{K_N}^-(G,X)$ have the same generic fibre. Since $\Phi_N$ is a finite morphism, $\Phi_N^-$ is quasi-finite and proper, hence $\Phi_N^-$ is also a finite morphism. Since $\mathfrak{M}_N$ is smooth 
by \cite{Rapoport} (resp.~\cite{Kottwitz})
, in particular it is normal, and thus $\Phi_N^-$ is a finite birational morphism from a normal scheme. 
Thus the normalization map $\nu: \mathscr{S}_{K_N}(G,X)\to \mathscr{S}_{K_N}^-(G,X)$ factors through $\mathfrak{M}_N$, and the map $\mathscr{S}_{K_N}(G,X)\to \mathfrak{M}_N$ is finite birational with a normal target $\mathfrak{M}_N$, by Zariski's main theorem this map has to be an isomorphism, 
and thus we have
\[\mathfrak{M}_N\cong \mathscr{S}_{K_N}(G,X)\]
i.e. Rapoport's model agrees with Kisin's model. The same argument shows that Kottwitz integral model $\mathcal{S}_K(G,X)$, which is also smooth, agrees with Kisin's integral model $\mathscr{S}_K(G,X)$. Therefore Rapoport, Kottwitz and Kisin models are all compatible with one another. 
\end{proof}
Note that all of the results in the following sections apply to both Rapoport's Hilbert modular integral models \cite{Rapoport} and Kottwitz's PEL type integral models\cite{Kottwitz}.

\subsection{Hodge embedding at infinite level}
Consider the morphism 
$\Phi_{\infty}^{(p)}: \mathcal{S}_{\infty}^{(p)}(G,X)\to \mathscr{S}_{K_{\infty}}^{(p)}(\GSp,S^{\pm})$, 
where on the level of $S$-points, for a $\Spec\Oo_{E,(v)}$-scheme $S$, it is given by ``forgetting endomorphism structures''
\begin{align*}
    \Phi_{\infty}(S):\mathscr{S}^{(p)}_{\infty}(G,X)(S)&\to \mathscr{S}_{K_{\infty}}^{(p)}(\GSp,S^{\pm})(S)\\
    (\mcA/S,\lambda,\iota,\eta_{\infty}^{(p)})&\mapsto (\mcA/S,\lambda,\eta_{\infty}^{(p)})
\end{align*}
where the left-hand side has infinite level structure given by
\[\eta_{\infty}^{(p)}:\Lambda\otimes \A_f^p\xrightarrow{\cong}(\underset{p\nmid N}{\varprojlim}\mcA[N])\otimes \A_f^p= \big(\prod\limits_{\ell\neq p} T_{\ell}(\mcA)\big)\otimes\A_f^p\]
\begin{lem}\label{finendo}
(i) Let $\phi,\phi'\in\End(\mcA)$. If for a fixed integer $N$, we have\\
$\phi|_{\mcA[N]}=\phi'|_{\mcA[N]}$, 
then $\phi=\phi'\mod N\End(\mcA/S)$.

(ii) Let $\phi,\phi'\in\End(\mcA)$. If for infinitely many integers $N$, we have\\
$\phi|_{\mcA[N]}=\phi'|_{\mcA[N]}$, 
Then $\phi=\phi'$ as endomorphisms of $\mcA$.
\end{lem}
\begin{proof}
(i) We have $\phi-\phi'|_{\mcA[N]}=0$, i.e. $\ker(\phi-\phi')\supset\mcA[N]=\ker([N])$, thus the map $\phi-\phi'$ factors through the multiplication by $N$ map $[N]$.
Thus $\phi-\phi'\in N\End(\mcA/S)$.

(ii) For infinitely many integers $N$, we have $\phi-\phi'|_{\mcA[N]}=0$. Thus the map $\phi-\phi'$ factors through the multiplication by $N$ map $[N]$ for infinitely many integers $N$, i.e. $\phi-\phi'\in N\End(\mcA/S)$ for infinitely many integers $N$. Now, $\End(\mcA/S)\cong\Z^m$ as an abelian group (for some $m$), $\phi-\phi'$ as an element of it is divisible by infinitely many integers $N$, thus we must have $\phi-\phi'=0$, i.e. $\phi=\phi'$ as elements of $\End(\mcA/S)$.
\end{proof}

\begin{lem}\label{infendo2}
There is a unique injective ring homomorphism 
\[\gamma:\End_S(\mcA)\otimes_{\Z}\Z_{(p)}\hookrightarrow\End_{}\Big(\big(\prod\limits_{\ell\neq p} T_{\ell}(\mcA)\big)\otimes\A_f^p\Big)\]
\end{lem}
\begin{proof}
For any $\phi\in\End(\mcA/S)$, it uniquely induces an endomorphism on the $\ell$-adic Tate module in the following way: by restriction we have a unique map $\phi|_{\mcA[\ell^n](\overline{k})}:\mcA[\ell^n](\overline{k})\xrightarrow{\phi[\ell^n]}\mcA[\ell^n](\overline{k})$, thus taking the inverse limit over $n$ on both sides we get the unique induced map $\gamma_{\ell}(\phi):T_{\ell}(\mcA)\to T_{\ell}(\mcA)$. Thus we indeed have a well-defined component map $\gamma_{\ell}: \End(\mcA/S)\to \End_{\Z_{\ell}}T_{\ell}(A)$ given by $\phi\mapsto \gamma_{\ell}(\phi)$.\\
To show that $\gamma_{\ell}$ is indeed injective: if $\gamma_{\ell}(\phi)=\gamma_{\ell}(\phi')$, then for infinitely many $n$, we have $\phi|_{\mcA[\ell^n]}=\phi'|_{\mcA[\ell^n]}$, thus by Lemma \ref{finendo}, we have that $\phi=\phi'$ as elements of $\End(\mcA/S)$. Thus $\gamma_{\ell}$ is injective. Taking the product over all $\ell\neq p$, we indeed have a well-defined injective map 
\[\prod\limits_{\ell\neq p}\gamma_{\ell}:\End(\mcA/S)\hookrightarrow\End_{}(\prod\limits_{\ell\neq p}T_{\ell}(\mcA))\]
We tensor the above map by $\Z_{(p)}$ and get
\begin{equation}\label{PEL-whatever}
    \prod\limits_{\ell\neq p}\gamma_{\ell}\otimes_{\Z}\Z_{(p)}:\End(\mcA/S)\otimes_{\Z}\Z_{(p)}\hookrightarrow\End_{}\Big(\prod\limits_{\ell\neq p}T_{\ell}(\mcA)\Big)\otimes_{\Z}\Z_{(p)}
\end{equation}
On the other hand, since 
$\A^p_f = \hat{\Z}^p \otimes_{\Z} \Z_{(p)}$, 
we have
\[\End_{\hat{\Z}^p}\Big(\prod\limits_{\ell\neq p}T_{\ell}(\mcA)\Big)\otimes_{\hat{\Z}^p}\A_f^p\cong\End_{}\Big(\prod\limits_{\ell\neq p}T_{\ell}(\mcA)\Big)\otimes_{\hat{\Z}^p}\hat{\Z}^p\otimes_{\Z}\Z_{(p)}\cong \End_{}\Big(\prod\limits_{\ell\neq p}T_{\ell}(\mcA)\Big)\otimes_{\Z}\Z_{(p)}\]
which is the right-hand side of (\ref{PEL-whatever}). On the other hand, we also have 
\[\End_{}\Big(\prod\limits_{\ell\neq p}T_{\ell}(\mcA)\Big)\otimes_{\hat{\Z}^p}\A_f^p\cong \End_{\A_f^p}\Big(\prod\limits_{\ell\neq p}T_{\ell}\mcA)\otimes_{\hat{\Z}^p}\A_f^p\Big)\]
Thus (\ref{PEL-whatever}) gives us
\[\prod\limits_{\ell\neq p}\gamma_{\ell}\otimes_{\Z}\Z_{(p)}:\End(\mcA/S)\otimes_{\Z}\Z_{(p)}\hookrightarrow\End_{}\Big(\prod\limits_{\ell\neq p}T_{\ell}(\mcA)\Big)\otimes_{\Z}\Z_{(p)}\cong \End_{\A_f^p}\Big((\prod\limits_{\ell\neq p}T_{\ell}\mcA)\otimes_{\hat{\Z}^p}\A_f^p\Big)\]
which is our desired $\gamma$.
\end{proof}
\begin{lem}\label{infendo} (i) The endomorphism structure $\iota:\Oo_B\hookrightarrow\End_S(\mcA)\otimes_{\Z}\Z_{(p)}$ induces a unique endomorphism structure on $V^{(p)}(\mcA)=(\prod\limits_{\ell\neq p} T_{\ell}(\mcA))\otimes \A_f^p$, i.e. $\iota$ uniquely induces an injective ring homomorphism \[\tilde{\iota}:\prod\limits_{\ell\neq p}\Oo_{B,\ell}\hookrightarrow\End\big(\prod\limits_{\ell\neq p} T_{\ell}(\mcA)\otimes\A_f^p\big)\]
(ii) Moreover, if $\tilde{\iota}=\tilde{\iota'}$, then $\iota=\iota'$. \\
i.e. the association $\iota\mapsto\tilde{\iota}$ is a one-to-one correspondence. 
\end{lem}
\begin{proof}
(i) First note that, since $\Z_{\ell}$ is flat over $\Z$, tensoring with $\Z_{\ell}$ preserves injective maps. Thus we tensor $\iota$ with $\Z_{\ell}$ and obtain:
\[\iota\otimes_{\Z}\Z_{\ell}: \Oo_{B,\ell}:=\Oo_B\otimes_{\Z}\Z_{\ell}\hookrightarrow\End_S(\mcA)\otimes_{\Z}\Z_{(p)}\otimes_{\Z}\Z_{\ell}\]
On the other hand, recall the well-known result
\[\End(\mcA/S)\otimes_{\Z}\Z_{\ell}\hookrightarrow \End_{\Z_{\ell}}(T_{\ell}\mcA)\]
Taking the product of the above map over all $\ell\neq p$, and we get
\begin{equation}\label{1stEnd}\prod\limits_{\ell\neq p}\big(\End(\mcA/S)\otimes_{\Z}\Z_{\ell}\big)\hookrightarrow \prod\limits_{\ell\neq p}\big(\End_{\Z_{\ell}}(T_{\ell}\mcA)\big)\end{equation}
Now, since $\End(\mcA/S)$ is finitely generated as a $\Z$-module, the tensor product on the left-hand-side commutes with the product on the left-hand-side, thus the left-hand-side of the above map (\ref{1stEnd}) becomes 
$\End(\mcA/S)\otimes_{\Z}\big(\prod\limits_{\ell\neq p}\Z_{\ell}\big)$. 
On the other hand, the right-hand-side of the map (\ref{1stEnd}) is 
$\prod\limits_{\ell\neq p}\big(\End_{\Z_{\ell}}(T_{\ell}\mcA)\big)\cong \End_{\prod\limits_{\ell\neq p}\Z_{\ell}}\big(\prod\limits_{\ell\neq p}T_{\ell}(\mcA)\big)$, 
where the equality holds because $\Hom_{\hat{\Z}}(T_{\ell}(\mcA),T_{\ell'}(\mcA))=0$ for $\ell\neq\ell'$. 
Thus 
\[\beta:\End(\mcA/S)\otimes_{\Z}\big(\prod\limits_{\ell\neq p}\Z_{\ell}\big)\hookrightarrow\End_{\prod\limits_{\ell\neq p}\Z_{\ell}}\big(\prod\limits_{\ell\neq p}T_{\ell}(\mcA)\big)\]
Thus we have obtained
\[\tilde{\iota}=\beta\circ (\prod\limits_{\ell\neq p}\iota\otimes_{\Z}\Z_{\ell}):\prod\limits_{\ell\neq p}\Oo_{L,\ell}\hookrightarrow \End(\mcA/S)\otimes_{\Z}\big(\prod\limits_{\ell\neq p}\Z_{\ell}\big)\hookrightarrow\End_{\prod\limits_{\ell\neq p}\Z_{\ell}}\big(\prod\limits_{\ell\neq p}T_{\ell}(\mcA)\big)\]
Note that this map $\beta$ is uniquely defined as we did not make any choice in the above construction.

(ii) If $\tilde{\iota}=\tilde{\iota'}$, then by uniqueness and injectivity of $\beta$, we have 
\[\prod\limits_{\ell\neq p}\iota\otimes_{\Z}\Z_{\ell}=\prod\limits_{\ell\neq p}\iota'\otimes_{\Z}\Z_{\ell}\]
Since $\Z_{\ell}$ is a flat $\Z$-module, we have that $\iota=\iota'$.
\end{proof}

\begin{prop}\label{injectivityprop-PEL}
The map $\Phi_{\infty}(S)$ is injective for any 
$\Oo_{E,(v)}$-scheme $S$.
\end{prop}
\textit{Idea of the Proof:} Injectivity of $\Phi_{\infty}(S)$ is equivalent to: take any $(\mcA/S,\lambda,\eta_{\infty}^{(p)})\in\Image(\Phi_{\infty}(S))$, there is only one preimage $(\mcA/S,\iota,\lambda,\eta_{\infty}^{(p)})\in\mathscr{S}^{(p)}_{\infty}(G,X)(S)$. Since the map $\Phi_{\infty}(S)$ is simply forgetting the endomorphism structure $\iota$, this is equivalent to saying that, given any $(\mcA/S,\lambda,\eta_{\infty})\in\Image(\Phi_{\infty}(S))$, there is only one possible endomorphism structure $\iota:\Oo_B\hookrightarrow\End(\mcA/S)\otimes_{\Z}\Z_{(p)}$ compatible with the other structures.

Suppose $(\mcA/S,\iota,\lambda,\eta_{\infty}^{(p)})\in\mathscr{S}^{(p)}_{\infty}(G,X)(S)$ is one pre-image point. We simply need to show that for any $\iota': \Oo_B\hookrightarrow\End(\mcA/S)\otimes_{\Z}\Z_{(p)}$ such that $\iota'\neq\iota$, we have $(\mcA/S,\iota',\lambda,\eta_{\infty}^{(p)})\notin\mathscr{S}^{(p)}_{\infty}(G,X)(S)$. The key condition that will fail in the moduli functor $\mathscr{S}^{(p)}_{\infty}(G,X)$ is the requirement that the level structure map $\eta_{\infty}^{(p)}$ has to be $\Oo_B$-linear. Details as follows.

\begin{proof}
First we consider the $\Oo_B$-action on $\big(\prod\limits_{\ell\neq p} T_{\ell}(\mcA)\big)\otimes\A_f^p$. Now, by Lemma \ref{infendo2}, $\iota:\Oo_B\hookrightarrow\End_S(\mcA)\otimes_{\Z}\Z_{(p)}$ induces a unique action of $\Oo_B$ on $\big(\prod\limits_{\ell\neq p} T_{\ell}(\mcA)\big)\otimes\A_f^p$ 
via composition $\gamma\circ\iota$ with the unique embedding
\[\gamma: \End_S(\mcA)\otimes_{\Z}\Z_{(p)}\hookrightarrow\End_{}\Big(\big(\prod\limits_{\ell\neq p} T_{\ell}(\mcA)\big)\otimes\A_f^p\Big).\] 
Recall that the infinite level map is given by
\[\eta_{\infty}^{(p)}:\Lambda\otimes \A_f^p\xlongrightarrow{\cong} \big(\prod\limits_{\ell\neq p} T_{\ell}(\mcA)\big)\otimes\A_f^p\]
and it satisfies the $\Oo_B$-linearity (or ``$\Oo_B$-compatibility'' if one prefers) condition: $\forall a\in\Oo_B$ and $\forall x\in \Lambda\otimes \A_f^p$, we have
\begin{equation}\label{pf2}\eta_{\infty}^{(p)}(a\cdot_B x)=a\underset{\gamma\circ\iota}{\cdot}\eta_{\infty}^{(p)}(x)=(\gamma\circ\iota(a))(\eta_{\infty}^{(p)}(x))\end{equation}
Now, for any $\iota'\neq\iota$, if $\eta_{\infty}^{(p)}$ still satisfied $\Oo_B$-linearity upon imposing $\iota'$ endomorphism structure, we would also have for $\forall a\in\Oo_B$ and $\forall x\in \Lambda\otimes \A_f^p$,
\begin{equation}\eta_{\infty}(a\cdot_B x)=a\underset{\gamma\circ\iota'}{\cdot}\eta_{\infty}^{(p)}(x)=(\gamma\circ\iota'(a))(\eta_{\infty}^{(p)}(x))\end{equation}
On the other hand, we combine it with equation \ref{pf2} to get: for $\forall a\in \Oo_B$ and $\forall x\in \Lambda\otimes \A_f^p$ (obviously we take the same $a\in\Oo_B$ and $x\in\Lambda\otimes \A_f^p$ on both sides)
\[(\gamma\circ\iota(a))(\eta_{\infty}^{(p)}(x))=(\gamma\circ\iota'(a))(\eta_{\infty}^{(p)}(x))\]
Since this holds for all $x\in \Lambda\otimes \A_f^p$, thus $\eta_{\infty}^{(p)}(x)$ ranges through all of \[\big(\prod\limits_{\ell\neq p} T_{\ell}(\mcA)\big)\otimes\A_f^p\]
and since $\gamma\circ\iota(a),\gamma\circ\iota'(a)\in \End_{}\Big(\big(\prod\limits_{\ell\neq p} T_{\ell}(\mcA)\big)\otimes\A_f^p\Big)$, thus we have: for all $a\in \Oo_B$,
\[(\gamma\circ\iota)(a)=(\gamma\circ\iota')(a)\]
Since $\gamma$ is injective by Lemma \ref{infendo2}, the above implies that $\iota(a)=\iota'(a)$ for all $a\in\Oo_B$, thus we have $\iota=\iota'$, contradiction. Thus $(\mcA/S,\iota',\lambda,\eta_{\infty}^{(p)})\notin\mathscr{S}^{(p)}_{\infty}(G,X)(S)$, and hence the injectivity of $\Phi_{\infty}(S)$.

Thus we have shown the uniqueness of endomorphism structure $\iota$ (should such a structure exist) in the case of infinite level structure; i.e. if $(\mcA/S,\iota,\lambda,\eta_{\infty}^{(p)})\in\mathscr{S}^{(p)}_{\infty}(G,X)(S)$ is a preimage point of some $(\mcA/S,\lambda,\eta_{\infty})\in\Image(\Phi(S))$, then for any $\iota'\neq\iota$, we must have $(\mcA/S,\iota',\lambda_R,\eta_{\infty})\notin\mathscr{S}^{(p)}_{\infty}(G,X)(S)$. Thus the map $\Phi_{\infty}(S)$ is injective. 
\end{proof}

\begin{remark}
Although we assume $\Oo_B$ is a maximal order of $B$ (a condition imposed in \cite{Kottwitz}) in Proposition \ref{injectivityprop-PEL}, the result holds for any order $\Oo$ of $B$ as we do not use the maximality of $\Oo_B$ in the proof \footnote{Should one like to consider an exotic PEL type model with non-maximal order action--such models do occur in nature sometimes}. 
\end{remark}

Since $\mathcal{S}_{K^p}$ (at finite level) is a quasi-projective 
$\Oo_{E,(v)}$-scheme by \cite[$\mathsection$ 5]{Kottwitz}, it is in particular locally Noetherian. 

\begin{remark}
For any $N\geq 1$, the diagram
\[\begin{tikzcd}
\mathscr{S}_{K_{N+1}}(G,X)\arrow[]{r}{}\arrow[]{d}{}&\mathscr{S}_{K_{N+1}}(\GSp,S^{\pm})\arrow[]{d}{}\\
\mathscr{S}_{K_N}(G,X)\arrow[]{r}{}&\mathscr{S}_{K_N}(\GSp,S^{\pm})
\end{tikzcd}\]
is not necessarily a fibre product diagram.
\end{remark}

The following proof applies to any PEL type integral models from \cite{Kottwitz}\footnote{The same method applies also to the PEL models in \cite{Rapoport-Zink}, but our expositions shall focus on the Kottwitz models.}.  
In particular, it applies to the Hilbert modular integral model \cite{Rapoport}. 
\begin{lem}\label{proper-m-infinity}
$\Phi_{\infty}$ is a proper morphism of 
$\Oo_{E,(v)}$-schemes.
\end{lem}
\begin{proof} We do it in several steps.

\textit{Step 1. We show that 
$\mathscr{S}_{K_{MN}}(G,X)\xrightarrow{f}\mathscr{S}_{K_M}(G,X)\times_{\mathscr{S}_{K'_M}(\GSp,S^{\pm})}\mathscr{S}_{K'_{MN}}(\GSp,S^{\pm})$ 
is a proper morphism of schemes. }

To see this, first we note that the composition map
\[\mathscr{S}_{K_{MN}}(G,X)\xrightarrow{f}\mathscr{S}_{K_M}(G,X)\times_{\mathscr{S}_{K_M}(\GSp,S^{\pm})}\mathscr{S}_{K_{MN}}(\GSp,S^{\pm})\xrightarrow{g} \mathscr{S}_{K_{MN}}(\GSp,S^{\pm})\]
is proper 
by Lemma \ref{Kottwitz-properness-finite} (essentially the argument of \cite[$\mathsection$5]{Kottwitz}). Note that the map $\Phi_{MN}: \mathscr{S}_{K_{MN}}(G,X)\to \mathscr{S}_{K_{MN}}(\GSp,S^{\pm})$ indeed factors through the fibre product $\mathscr{S}_{K_M}(G,X)\times_{\mathscr{S}_{K_M}(\GSp,S^{\pm})}\mathscr{S}_{K_{MN}}(\GSp,S^{\pm})$ by the universal property of fibre products, and moreover the factorization map $f$ is unique. 

On the other hand, since the map 
$\Phi_N:\mathcal{S}_{K_N}(G,X)\to\mathscr{S}_{K_N}(\GSp,S^{\pm})$ 
is proper again by Lemma \ref{Kottwitz-properness-finite} (the standard argument of \cite[$\mathsection$ 5]{Kottwitz}), it is in particular separated. We view it as a morphism of schemes over $\mathscr{S}_{K_N}(\GSp,S^{\pm})$, and consider the base change of the morphism $\Phi_N$ along the (transition) morphism 
\[\mathscr{S}_{K_{MN}}(\GSp,S^{\pm})\to\mathscr{S}_{K_N}(\GSp,S^{\pm}),\]
thus its base change morphism 
\begin{align*}
\Phi_{K_N}\times_{\Id_{\mathscr{S}_{K_N}}(\GSp,S^{\pm})}\Id_{\mathscr{S}_{K_{MN}}(\GSp,S^{\pm})}&: \mathscr{S}_{K_N}(G,X)\times_{\mathscr{S}_{K_N}(\GSp,S^{\pm})}\mathscr{S}_{K_{MN}}(\GSp,S^{\pm})\\
&\to \mathscr{S}_{K_N}(\GSp,S^{\pm})\times_{\mathscr{S}_{K_N}(\GSp,S^{\pm})}\mathscr{S}_{K_{MN}}(\GSp,S^{\pm})\\
&=\mathscr{S}_{K_{MN}}(\GSp,S^{\pm})
\end{align*}
is also proper, thus separated, and this is precisely our morphism $g$. 
Now, since the composition $g\circ f$ is proper, and $g$ is separated, then $f$ is proper.

\textit{Step 2. We show that $\mathscr{S}_{K_{MN}}(G,X)\xrightarrow{f}\mathscr{S}_{K_M}(G,X)\times_{\mathscr{S}_{K_M}(\GSp,S^{\pm})}\mathscr{S}_{K_{MN}}(\GSp,S^{\pm})$ 
is a monomorphism in the category of schemes. }

To check this, we look at $f$ on the level of $S$-points, and we have
\begin{align*}
f(S):\mathscr{S}_{K_{MN}}(G,X)(S)&\to\mathscr{S}_{K_M}(G,X)(S)\times_{\mathscr{S}_{K_M}(\GSp,S^{\pm})(S)}\mathscr{S}_{K_{MN}}(\GSp,S^{\pm})(S)\\
\big(\mcA/S,\lambda,\iota,K_{MN}^p\eta^p\big)&\mapsto \Bigg(\big(\mcA/S,\lambda,\iota,K_M^p\eta^p\big),\big(\mcA/S,\lambda,K'^p_{MN}\eta'^p\big)\Bigg)
\end{align*}
To see that $f(S)$ is injective: given any point on the RHS, the first component 
$\big(\mcA/S,\lambda,\iota,\mcA[M]\xrightarrow{\sim}(\Oo_B/M\Oo_B)^2\big)$ 
fixes our 
$(\mcA/S,\iota,\lambda)$ 
and moreover, the second component also fixes an isomorphism 
$\mcA[MN]\xrightarrow{\sim}(\Z/MN\Z)^{2g}$, 
which uniquely fixes an isomorphism, by multiplication by $N$:
$\mcA[M]\xrightarrow{\sim}(\Z/M\Z)^{2g}$. 
Note that $(\Z/M\Z)^{2g}\xrightarrow{\sim}(\Oo_B/M\Oo_B)^2$ 
uniquely lifts to a map 
$(\Z/MN\Z)^{2g}\xrightarrow{\sim}(\Oo_B/MN\Oo_B)^2$, 
which then picks out a unique level $MN$-structure for $\mathscr{S}_K(G,X)$. 
Thus given any point on the right-hand side, we can uniquely pick out a point on the left-hand side which maps to the point on the right-hand side, so the morphism $f$ is a monomorphism.

\textit{Step 3: 
Steps 1 and 2 imply that 
\begin{equation}\label{MN-base-change}
    \mathscr{S}_{K_{MN}}(G,X)\xrightarrow{f}\mathscr{S}_{K_M}(G,X)\times_{\mathscr{S}_{K_M}(\GSp,S^{\pm})}\mathscr{S}_{K_{MN}}(\GSp,S^{\pm})
\end{equation}
is a closed embedding.}

Since $f$ is a closed embedding, the base change of $f$, along the morphism 
$\mathscr{S}_{K_{\infty}}^{(p)}(\GSp,S^{\pm})\to \mathscr{S}_{K_{MN}}(\GSp,S^{\pm})$ 
gives a closed embedding:
\begin{align*}\mathscr{S}_{K_{MN}}(G,X)\underset{{\mathscr{S}_{K_{MN}}(\GSp,S^{\pm})}}{\times}\mathscr{S}_{K_{\infty}}^{(p)}(\GSp,S^{\pm})&\xrightarrow{f}\mathscr{S}_{K_M}(G,X)\underset{{\mathscr{S}_{K_M}(\GSp,S^{\pm})}}{\times}\mathscr{S}_{K_{MN}}(\GSp,S^{\pm})\underset{{\mathscr{S}_{K_{MN}}}}{\times}\mathscr{S}_{K_{\infty}}^{(p)}(\GSp
)\\
&\cong \mathscr{S}_{K_M}(G,X)\times_{\mathscr{S}_{K_M}(\GSp,S^{\pm})}\mathscr{S}_{K_{\infty}}^{(p)}(\GSp,S^{\pm})
\end{align*}
We denote the above obtained map as
\[\tilde{f}_{N,1}: \mathscr{S}_{K_{MN}}(G,X)\times_{\mathscr{S}_{K_{MN}}(\GSp,S^{\pm})}\mathscr{S}_{K_{\infty}}^{(p)}(\GSp,S^{\pm})\to \mathscr{S}_{K_M}(G,X)\times_{\mathscr{S}_{K_M}(\GSp,S^{\pm})}\mathscr{S}_{K_{\infty}}^{(p)}(\GSp,S^{\pm})\]
which is a closed embedding.

\textit{Step 4.} We take the inverse limit of the morphism $\tilde{f}_N$ over all $N$, i.e. consider
\[\tilde{f}_{\infty}:\underset{N}{\varprojlim}\big(\mathscr{S}_{K_{MN}}(G,X)\times_{\mathscr{S}_{K_{MN}}(\GSp,S^{\pm})}\mathscr{S}_{K_{\infty}}^{(p)}(\GSp,S^{\pm})\big)\to \underset{N}{\varprojlim}\mathscr{S}_{K_M}(G,X)\times_{\mathscr{S}_{K_M}(\GSp,S^{\pm})}\mathscr{S}_{K_{\infty}}^{(p)}(\GSp,S^{\pm})\]
Since taking inverse limit commutes with taking fibre product, the left-hand side of the map $\tilde{f}_{\infty}$ is: 
\begin{align*}
\underset{N}{\varprojlim}\big(\mathscr{S}_{K_{MN}}(G,X)\times_{\mathscr{S}_{K_{MN}}(\GSp,S^{\pm})}\mathscr{S}_{K_{\infty}}^{(p)}(\GSp,S^{\pm})\big)&=\big(\underset{N}{\varprojlim}\mathscr{S}_{K_{MN}}(G,X)\big)\times_{\underset{N}{\varprojlim}\mathscr{S}_{K_{MN}}(\GSp,S^{\pm})}\mathscr{S}_{K_{\infty}}(\GSp,S^{\pm})\\
&=\mathscr{S}_{\infty}^{(p)}(G,X)\times_{\mathscr{S}_{K_{\infty}}(\GSp,S^{\pm})}\mathscr{S}_{K_{\infty}}(\GSp,S^{\pm})\cong \mathscr{S}_{\infty}^{(p)}(G,X)
\end{align*}
whereas the right-hand side of the map $\tilde{f}_{\infty}$ stays the same as it does not depend on $N$ at all. 
Thus $\tilde{f}_{\infty}=\underset{N}{\varprojlim}\tilde{f}_N$ becomes
\[\tilde{f}_{\infty}: \mathscr{S}_{\infty}^{(p)}(G,X)\to \mathscr{S}_{K_M}(G,X)\times_{\mathscr{S}_{K_M}(\GSp,S^{\pm})}\mathscr{S}_{K_{\infty}}(\GSp,S^{\pm}).\] 
Since each $\tilde{f}_N$ is a closed embedding, and moreover the transition maps are closed embeddings, $\tilde{f}_{\infty}$ is a closed embedding. 
We shall give a full proof for this claim in Lemma \ref{limit-closed-embedding}.

\textit{Step 4. We show that the projection morphism
\[\psi_M: \mathscr{S}_{K_M}(G,X)\times_{\mathscr{S}_{K_M}(\GSp,S^{\pm})}\mathscr{S}_{K_{\infty}}^{(p)}(\GSp,S^{\pm})\to \mathscr{S}_{K_{\infty}}^{(p)}(\GSp,S^{\pm})\]
is proper.} 

This is simply because: 
$\Phi_M: \mathscr{S}_{K_M}(G,X)\to\mathscr{S}_{K_M}(\GSp,S^{\pm})$ 
is proper (by Lemma \ref{Kottwitz-properness-finite}), so when we take base change of the morphism $\Phi_M$ along the projection morphism
$\mathscr{S}_{K_{\infty}}^{(p)}(\GSp,S^{\pm})\to \mathscr{S}_{K_M}(\GSp,S^{\pm})$ 
and consider 
\begin{align*}
\psi_M:\mathscr{S}_{K_M}(G,X)\times_{\mathscr{S}_{K_M}(\GSp,S^{\pm})}\mathscr{S}_{K_{\infty}}(\GSp,S^{\pm})
&\to\mathscr{S}_{K_M}(\GSp,S^{\pm})\times_{\mathscr{S}_{K_M}(\GSp,S^{\pm})}\mathscr{S}_{K_{\infty}(\GSp,S^{\pm})}\\
&=\mathscr{S}_{K_{\infty}}(\GSp,S^{\pm})
\end{align*}
The above morphism 
$\psi_M=\Phi_M\times_{\Id_{\mathscr{S}_{K_M}}(\GSp,S^{\pm})}\Id_{\mathscr{S}_{K_{\infty}}(\GSp,S^{\pm})}$, 
as the base change of a proper morphism $\Phi_M$, is also proper.

\textit{Step 5.} Finally, the morphism 
$\Phi_{\infty}^{(p)}:\mathscr{S}_{\infty}^{(p)}(G,X)\to\mathscr{S}_{K_{\infty}}^{(p)}(\GSp,S^{\pm})$, 
which, as a composition of proper morphisms
\[\tilde{f}_{\infty}:\mathscr{S}_{\infty}^{(p)}(G,X)\to \mathscr{S}_{K_M}(G,X)\times_{\mathscr{S}_{K_M}(\GSp,S^{\pm})}\mathscr{S}_{K_{\infty}}^{(p)}(\GSp,S^{\pm})\]
and 
$\psi_M:\mathscr{S}_{K_M}(G,X)\times_{\mathscr{S}_{K_M}(\GSp,S^{\pm})}\mathscr{S}_{K_{\infty}}^{(p)}(\GSp,S^{\pm})\to\mathscr{S}_{K_{\infty}}^{(p)}(\GSp,S^{\pm})$, 
is proper. (Note that the map $\Phi_{\infty}^{(p)}$ indeed factors through the map $\tilde{f}_{\infty}$ by the universal property of fibre products, and moreover the map $\tilde{f}_{\infty}$ is unique.) 
Now $\Phi_{\infty}=\psi_M\circ \tilde{f}_{\infty}$ as a composition of proper morphisms is proper.
\end{proof}

We finish up the above proof by the following lemma.
\begin{lem}\label{limit-closed-embedding}
$\tilde{f}_{\infty}$ is a closed embedding.
\end{lem}
\begin{proof}
Recall our finite step maps
\[\tilde{f}_{N,1}: \mathscr{S}_{K_{MN}}(G,X)\times_{\mathscr{S}_{K_{MN}}(\GSp,S^{\pm})}\mathscr{S}_{K_{\infty}}^{(p)}(\GSp,S^{\pm})\to \mathscr{S}_{K_M}(G,X)\times_{\mathscr{S}_{K_M}(\GSp,S^{\pm})}\mathscr{S}_{K_{\infty}}^{(p)}(\GSp,S^{\pm})\]
and we have an inverse limit
\[\tilde{f}_{\infty}:=\underset{N}{\varprojlim}\tilde{f}_{N,1}:\underset{N}{\varprojlim}\big(\mathscr{S}_{K_{MN}}(G,X)\underset{{\mathscr{S}_{K_{MN}}(\GSp,S^{\pm})}}{\times}\mathscr{S}_{K_{\infty}}^{(p)}(\GSp,S^{\pm})\big)\to \mathscr{S}_{K_M}(G,X)\underset{{\mathscr{S}_{K_M}(\GSp,S^{\pm})}}{\times}\mathscr{S}_{K_{\infty}}^{(p)}(\GSp,S^{\pm})\]
To show that $\tilde{f}_{\infty}$ is a closed embedding: take any affine open 
\[\Spec R=U\subset \mathscr{S}_{K_M}(G,X)\times_{\mathscr{S}_{K_M}(\GSp,S^{\pm})}\mathscr{S}_{K_{\infty}}^{(p)}(\GSp,S^{\pm})\]
and we have 
$\tilde{f}_{\infty}^{-1}(U)=\big(\underset{N}{\varprojlim}\tilde{f}_{N,1}\big)^{-1}(U)=\underset{N}{\varprojlim}\big(\tilde{f}_{N,1}^{-1}(U)\big)$.

Since each $\tilde{f}_{N,1}$ is a closed embedding, there exists an ideal $I_{N,1}\subset R$, for each $N$, such that 
$\tilde{f}_{N,1}^{-1}(U)=\Spec (R/I_{N,1})$ 
as schemes over $U=\Spec R$. Thus we have 
$\tilde{f}_{\infty}^{-1}(U)=\underset{N}{\varprojlim}\big(\tilde{f}_{N,1}^{-1}(U)\big)=\underset{N}{\varprojlim}\Spec(R/I_{N,1})$. 
To simplify the above formula, first we note that 
$\underset{N}{\varprojlim}\Spec(R/I_{N,1})=\Spec\big(\underset{N}{\varinjlim}(R/I_{N,1})\big)$.

Next we note that
$\underset{N}{\varinjlim}(R/I_{N,1})=R/\bigcup\limits_NI_{N,1}$. To see this: recall that we showed
\[\tilde{f}_{N,1}: \mathscr{S}_{K_{MN}}(G,X)\times_{\mathscr{S}_{K_{MN}}(\GSp,S^{\pm})}\mathscr{S}_{K_{\infty}}^{(p)}(\GSp,S^{\pm})\to \mathscr{S}_{K_M}(G,X)\times_{\mathscr{S}_{K_M}(\GSp,S^{\pm})}\mathscr{S}_{K_{\infty}}^{(p)}(\GSp,S^{\pm})\]
is a closed embedding, for any $M,N$. 
Now we plug in $M\to MN$ and $N\to q$, and still get a closed embedding
\[\tilde{f}_{Nq,N}: \mathscr{S}_{K_{MNq}}(G,X)\times_{\mathscr{S}_{K_{MNq}}(\GSp,S^{\pm})}\mathscr{S}_{K_{\infty}}^{(p)}(\GSp,S^{\pm})\to \mathscr{S}_{K_{MN}}(G,X)\times_{\mathscr{S}_{K_{MN}}(\GSp,S^{\pm})}\mathscr{S}_{K_{\infty}}^{(p)}(\GSp,S^{\pm})\]
Thus we have a series of closed embeddings: (indexing by divisibility)
\begin{align*}
\cdots\hookrightarrow\mathscr{S}_{K_{MNq}}(G,X)\times_{\mathscr{S}_{K_{MNq}}(\GSp,S^{\pm})}\mathscr{S}_{K_{\infty}}^{(p)}(\GSp,S^{\pm})&\xhookrightarrow{\tilde{f}_{Nq,N}}\mathscr{S}_{K_{MN}}(G,X)\times_{\mathscr{S}_{K_{MN}}(\GSp,S^{\pm})}\mathscr{S}_{K_{\infty}}^{(p)}(\GSp,S^{\pm})\\
&\xhookrightarrow{\tilde{f}_{N,1}} \mathscr{S}_{K_M}(G,X)\times_{\mathscr{S}_{K_M}(\GSp,S^{\pm})}\mathscr{S}_{K_{\infty}}^{(p)}(\GSp,S^{\pm})
\end{align*}
Also note that the transition maps in our system satisfy
$\tilde{f}_{Nq,1}=\tilde{f}_{N,1}\circ\tilde{f}_{Nq,N}$. 
Thus for any affine open $\Spec R=U\subset \mathscr{S}_{K_M}(G,X)\times_{\mathscr{S}_{K_M}(\GSp,S^{\pm})}\mathscr{S}_{K_{\infty}}^{(p)}(\GSp,S^{\pm})$, we have
$\tilde{f}_{N,1}^{-1}(U)=\Spec R/I_{N,1}$ 
as schemes over $U=\Spec R$, for some ideal $I_{N,1}\subset R$. On the other hand, since
$\tilde{f}_{Nq,1}$ is also a closed embedding, we have 
\[\tilde{f}_{Nq,1}^{-1}(U)=\Spec R/I_{Nq,1}\] 
for some ideal $I_{Nq,1}\subset R$. And since $\tilde{f}_{Nq,1}=\tilde{f}_{N,1}\circ\tilde{f}_{Nq,N}$, we have 
\[\Spec R/I_{Nq,1}=\tilde{f}_{Nq,1}^{-1}(U)=\tilde{f}_{Nq,N}^{-1}\circ\tilde{f}_{N,1}^{-1}(U)=\tilde{f}_{Nq,N}^{-1}(\Spec R/I_{N,1})\]
Since $\tilde{f}_{Nq,N}$ is also a closed embedding, and $\Spec R/I_{N,1}$ is an open affine of $\mathscr{S}_{K_{MN}}(G,X)\times_{\mathscr{S}_{K_{MN}}(\GSp,S^{\pm})}\mathscr{S}_{K_{\infty}}^{(p)}(\GSp,S^{\pm})$, thus 
\[\Spec R/I_{Nq,1}=\tilde{f}_{Nq,N}^{-1}(\Spec R/I_{N,1})=\Spec (R/I_{N,1})/I'\]
for some ideal $I'\subset R/I_{N,1}$, and thus 
$R/I_{Nq}=(R/I_N)/I'$. 
Thus as ideals of $R$, we have
$I_N\subset I_{Nq}$. 
Thus the ideals (ordered by divisibility of subscripts) $I_N\subset I_{Nq}\subset\cdots$ form an ascending chain of ideals in ring $R$, and thus 
\[\underset{N}{\varinjlim}(R/I_N)=R/\underset{N}{\varinjlim}I_N=R/\bigcup\limits_NI_N\]
And since we have the ascending chain of ideals, $\bigcup\limits_NI_N$ is indeed an ideal of $R$, and thus $\tilde{f}_{\infty}$ is indeed a closed embedding.
\end{proof}

We have now completed the proof of Lemma \ref{proper-m-infinity}. The following proposition is now an easy corollary. 
\begin{prop}\label{infinite-level-embedding}
$\Phi_{\infty}$ is a closed embedding. 
\end{prop}
\begin{proof}
By Proposition \ref{injectivityprop-PEL}, $\Phi_{\infty}^{(p)}$ is injective on the level of $S$-points for any 
$\Oo_{E,(v)}$-scheme $S$, thus 
$\Phi_{\infty}$ is a monomorphism as a morphism of 
$\Oo_{E,(v)}$-schemes. On the other hand, by Lemma \ref{proper-m-infinity}, $\Phi_{\infty}$ is also a proper morphism, thus it is a proper monomorphism, and thus it is a closed embedding. 
\end{proof}

\begin{remark}\label{non-maximal-order-remark}
Although we assume $\Oo_B$ is a maximal order of $B$ (a condition imposed in \cite{Kottwitz}) in Proposition \ref{infinite-level-embedding}, the result holds for any order $\Oo$ of $B$ as we do not use the maximality of $\Oo_B$ in the proof of Lemma \ref{proper-m-infinity}, where the cited result from \cite[$\mathsection$5]{Kottwitz} only uses the representability result from the theory of Hilbert schemes, and holds regardless of whether $\Oo$ is a maximal order or not\footnote{Again should one like to consider an exotic PEL type model with non-maximal order action}. 
\end{remark}

\subsection{Descent down to finite level}\label{descent-down-section}
\begin{Coro}\label{finite-level-embedding}
There exist $K\subset G(\A_f)$ and $K'\subset \GSp(\A_f)$ such that the map 
\[\Phi_{K,K'}:\mathscr{S}_K(G,X)\to \mathscr{S}_{K'}(\GSp,S^{\pm})\]
is a closed embedding. 
\end{Coro}
\begin{proof}
Consider the graph $\mathrm{Graph}_{K,K'}$ of the equivalence relation \[(x,y)\in\mathscr{S}_K(G,X)^2:\Phi_{K,K'}(x)=\Phi_{K,K'}(y),\]
which gives a closed subscheme of $\mathscr{S}_K(G,X)^2$. By definition, the graph $\mathrm{Graph}_{K,K'}$ shrinks as $K'\supset K$ shrinks. Since $\mathscr{S}_K(G,X)$ is locally Noetherian, $\mathrm{Graph}_{K,K'}$ stabilizes at some $K'\supset K$. By Proposition \ref{infinite-level-embedding}, $\Phi_{K,K'}$ is a closed embedding at some finite levels $K\subset G(\A_f)$ and $K'\subset\GSp(\A_f)$. 
\end{proof}

\begin{remark}
We could have deferred the proof of Lemma \ref{equiv-models} until now, using the stronger result of Corollary \ref{finite-level-embedding} instead of simply Lemma \ref{Kottwitz-properness-finite}. We chose the harder approach to emphasize that finiteness of the Hodge morphism suffices in deducing the equivalence of models. 
\end{remark}
Using Corollary \ref{finite-level-embedding}, the proof of Lemma \ref{equiv-models} now simplifies as follows: 
\begin{proof}
By Corollary \ref{finite-level-embedding}, we have a closed embedding $\mathcal{S}_K(G,X)\hookrightarrow\mathscr{S}_{K'}(\GSp,S^{\pm})$. 
Since $\mathcal{S}_K(G,X)$ is smooth by \cite{Kottwitz} (resp.~\cite{Rapoport}), in particular it is flat, and since $\mathscr{S}_K^-(G,X)_{\eta}\simeq \Sh_K(G,X)\simeq \mathcal{S}_K(G,X)_{\eta}$, we have $\mathcal{S}_K(G,X)\simeq\mathscr{S}_K^-(G,X)$. 
Since $\mathcal{S}_K(G,X)$ is normal, 
$\mathcal{S}_K(G,X)\simeq\mathscr{S}_K^-(G,X)\simeq \mathscr{S}_K(G,X)$, which is the desired equivalence of models.
\end{proof}

\bibliographystyle{amsalpha}
\bibliography{bibfile}

\end{document}